\documentclass[12pt,eqno,]{article}
\usepackage{amsmath,amsthm,amssymb}
\usepackage[a4paper]{geometry}
\theoremstyle{plain}
\newtheorem{theorem}{THEOREM}[section]

\newtheorem{lemma}[theorem]{LEMMA}

\numberwithin{theorem}{section}
\numberwithin{equation}{section}
\numberwithin{table}{section}
\numberwithin{figure}{section}

\theoremstyle{definition}

\usepackage{enumerate}
\usepackage{fancyhdr}
\makeatletter

\newcommand{\Rmnum}[1]{\expandafter\@slowromancap\romannumeral #1@}
\usepackage[center]{titlesec}
\makeatother

\title{\normalsize \bf SOME BEST POSSIBLE INEQUALITIES CONCERNING CERTAIN BIVARIATE MEANS}
\author{\small  \textsc{Tiehong Zhao, Yuming Chu and Baoyu Liu}}
\date{}

\begin{document}
\maketitle

\renewcommand{\thefootnote}{\fnsymbol{footnote}}

\footnotetext{\hspace*{-5mm}
\begin{tabular}{@{}r@{}p{14.0cm}@{}}
&\qquad {\it Mathematics Subject Classification (2010)}: 26E60.\\
&\qquad {\it Keywords and phrases}: Harmonic mean, geometric mean, Neuman-S\'{a}ndor mean, quadratic mean, contra-harmonic mean.\\
&\qquad This research was supported by the Natural Science
Foundation of China under Grants 11071069 and 11171307, and Innovation
Team Foundation of the Department of Education of Zhejiang Province
under Grant T200924.
\end{tabular}}

\vspace{-1cm}

\def\abstractname{}
\begin{abstract}
\noindent {\it Abstract.} In this paper, some inequalities of bounds for the Neuman-S\'{a}ndor mean in terms of weighted arithmetic means of two bivariate means are established. Bounds involving weighted arithmetic means are sharp.
\end{abstract}

\section{Introduction}

\medskip

For $a,b>0$ with $a\neq b$ the Neuman-S\'{a}ndor mean $M(a,b)$ [1] is defined by
\begin{equation*}
M(a,b)=\frac{a-b}{2{{\sinh^{-1}}}\left[(a-b)/(a+b)\right]},
\end{equation*}
where ${\sinh^{-1}}(x)=\log(x+\sqrt{1+x^2})$ is the inverse
hyperbolic sine function.

\medskip

Recently, the Neuman-S\'{a}ndor mean has been the subject of
intensive research. In particular, many remarkable inequalities for
the Neuman-S\'{a}ndor mean $M(a,b)$ can be found in the literature
[1-4].

\medskip

Let $H(a,b)=2ab/(a+b)$, $G(a,b)=\sqrt{ab}$, $L(a,b)=(b-a)/(\log
b-\log a)$, $P(a,b)=(a-b)/(4\arctan\sqrt{a/b}-\pi)$,
$A(a,b)=(a+b)/2$, $T(a,b)=(a-b)/[2\arcsin((a-b)/(a+b))]$,
$Q(a,b)=\sqrt{(a^2+b^2)/2}$ and $C(a,b)=(a^2+b^2)/(a+b)$ be the
harmonic, geometric, logarithmic, first Seiffert, arithmetic, second
Seiffert, quadratic and contra-harmonic means of $a$ and $b$,
respectively. Then it is well-known that the inequalities
\begin{align*}
H(a,b)<G(a,b)<L(a,b)<P(a,b)<A(a,b)\\ <M(a,b)<T(a,b)<Q(a,b)<C(a,b)
\end{align*}
hold for all $a,b>0$ with $a\neq b$.

\medskip

In [1, 2], Neuman and S\'{a}ndor proved that the double inequalities
\begin{equation*}
A(a,b)<M(a,b)<T(a,b),
\end{equation*}
\begin{equation*}
P(a,b)M(a,b)<A^{2}(a,b),
\end{equation*}
\begin{equation*}
A(a,b)T(a,b)<M^{2}(a,b)<(A^{2}(a,b)+T^{2}(a,b))/2
\end{equation*}
hold for all $a,b>0$ with $a\neq b$.

\medskip

Let $0<a,b<1/2$ with $a\neq b$, $a'=1-a$ and $b'=1-b$. Then the
following Ky Fan inequalities
\begin{equation*}
\frac{G(a,b)}{G(a',b')}<\frac{L(a,b)}{L(a',b')}<\frac{P(a,b)}{P(a',b')}<\frac{A(a,b)}{A(a',b')}<\frac{M(a,b)}{M(a',b')}<\frac{T(a,b)}{T(a',b')}
\end{equation*}
were presented in [1].

\medskip

The double inequality $L_{p_{0}}(a,b)<M(a,b)<L_{2}(a,b)$ for all
$a,b>0$ with $a\neq b$ was established by Li et al. in [3], where
$L_{p}(a,b)=[(b^{p+1}-a^{p+1})/((p+1)(b-a))]^{1/p}(p\neq -1, 0)$,
$L_{0}(a,b)=1/e(b^{b}/a^{a})^{1/(b-a)}$ and
$L_{-1}(a,b)=(b-a)/(\log{b}-\log{a})$ is the $p$-th generalized
logarithmic mean of $a$ and $b$, and $p_{0}=1.843\cdots$ is the
unique solution of the equation $(p+1)^{1/p}=2\log(1+\sqrt{2})$.

\medskip

Neuman [4] proved that the double inequalities
$$\alpha Q(a,b)+(1-\alpha)A(a,b)<M(a,b)<\beta Q(a,b)+(1-\beta)A(a,b)$$
and
$$\lambda Q(a,b)+(1-\lambda)A(a,b)<M(a,b)<\mu Q(a,b)+(1-\mu)A(a,b)$$
hold for all $a,b>0$ with $a\neq b$ if and only if $\alpha\leq(1-\log(\sqrt{2}+1))/[(\sqrt{2}-1)\log(\sqrt{2}+1)]=0.3249\cdots$, $\beta\geq1/3$,
$\lambda\leq(1-\log(\sqrt{2}+1))/\log(\sqrt{2}+1)=0.1345\cdots$ and $\mu\geq1/6$.

\medskip

The main purpose of this paper is to find the least values $\alpha_1,\alpha_2,\alpha_3,$ and the greatest values $\beta_1,\beta_2,\beta_3,$ such that the double inequalities
\begin{align*}
\alpha_1H(a,b)+(1-\alpha_1)Q(a,b)&<M(a,b)<\beta_1H(a,b)+(1-\beta_1)Q(a,b),\\
\alpha_2G(a,b)+(1-\alpha_2)Q(a,b)&<M(a,b)<\beta_2G(a,b)+(1-\beta_2)Q(a,b),\\
\alpha_3H(a,b)+(1-\alpha_3)C(a,b)&<M(a,b)<\beta_3H(a,b)+(1-\beta_3)C(a,b)
\end{align*}
hold true for all $a,b>0$ with $a\neq b$.

\medskip

Our main results are presented in Theorems 1.1-1.3.

\begin{theorem}
The double inequality
\begin{equation}
\alpha_1H(a,b)+(1-\alpha_1)Q(a,b)<M(a,b)<\beta_1H(a,b)+(1-\beta_1)Q(a,b)
\end{equation}
holds for all $a,b>0$ with $a\neq b$ if and only if $\alpha_1\geq 2/9=0.2222\cdots$ and $\beta_1\leq1-1/[\sqrt{2}\log(1+\sqrt{2})]=0.1977\cdots$.
\end{theorem}

\begin{theorem}
The double inequality
\begin{equation}
\alpha_2G(a,b)+(1-\alpha_2)Q(a,b)<M(a,b)<\beta_2G(a,b)+(1-\beta_2)Q(a,b)
\end{equation}
holds for all $a,b>0$ with $a\neq b$ if and only if $\alpha_2\geq1/3=0.3333\cdots$ and $\beta_2\leq1-1/[\sqrt{2}\log(1+\sqrt{2})]=0.1977\cdots$.
\end{theorem}

\begin{theorem}
The double inequality
\begin{equation}
\alpha_3H(a,b)+(1-\alpha_3)C(a,b)<M(a,b)<\beta_3H(a,b)+(1-\beta_3)C(a,b)
\end{equation}
holds for all $a,b>0$ with $a\neq b$ if and only if $\alpha_3\geq1-1/[2\log(1+\sqrt{2})]=0.4327\cdots$ and $\beta_3\leq5/12=0.4166\cdots$.
\end{theorem}

\section{Lemmas}

In order to prove our main results we need two Lemmas, which we
present in this section.

\begin{lemma} (See [5, Lemma 1.1]). Suppose that the power series
$f(x)=\sum\limits_{n=0}^{\infty}a_{n}x^{n}$ and
$g(x)=\sum\limits_{n=0}^{\infty}b_{n}x^{n}$ have the radius of
convergence $r>0$ and $b_{n}>0$ for all $n\in\{0,1,2,\cdots\}$. Let
$h(x)={f(x)}/{g(x)}$, then the following statements are true:
\begin{enumerate}[(1)]
\item If the sequence $\{a_{n}/b_{n}\}_{n=0}^{\infty}$ is (strictly)
increasing (decreasing), then $h(x)$ is also (strictly) increasing
(decreasing) on $(0,r)$;
\item If the sequence $\{a_{n}/b_{n}\}$ is (strictly) increasing
(decreasing) for $0<n\leq n_{0}$ and (strictly) decreasing
(increasing) for $n>n_{0}$, then there exists $x_{0}\in(0,r)$ such
that $h(x)$ is (strictly) increasing (decreasing) on $(0,x_{0})$ and
(strictly) decreasing (increasing) on $(x_{0},r)$.
\end{enumerate}
\end{lemma}

\begin{lemma}
Let $p\in(0,1)$, $\lambda_0=1-1/[\sqrt{2}\log(1+\sqrt{2})]=0.1977\cdots$ and
\begin{equation}\label{eqn:2.1}
f_p(x)=\sinh^{-1}(x)-\frac{x}{\sqrt{1+x^2}-p\left(\sqrt{1+x^2}-\sqrt{1-x^2}\right)}.
\end{equation}
Then $f_{1/3}(x)<0$ and $f_{\lambda_0}(x)>0$ for all $x\in(0,1)$.
\end{lemma}

\begin{proof}
From (\ref{eqn:2.1}) one has
\begin{align}
f_p(0)&=0,\\
f_p(1)&=\log(1+\sqrt{2})-\frac{1}{\sqrt{2}(1-p)},\\ \label{eqn:2.4}
f'_p(x)&=\frac{g_p(x)}{\sqrt{1-x^4}(\sqrt{1+x^2}+p(\sqrt{1-x^2}-\sqrt{1+x^2}))^2},
\end{align}
where
\begin{align}
\begin{split}
g_p(x)=&\sqrt{1-x^2}\left(\sqrt{1+x^2}+p(\sqrt{1-x^2}-\sqrt{1+x^2})\right)^2-\sqrt{1-x^2}\\
&\hspace{5.6cm}-p(\sqrt{1+x^2}-\sqrt{1-x^2}).
\end{split}
\end{align}

We divide the proof into two cases.
\setlength\leftmargini{.1cm}
\begin{description}
\item[Case 1]$p=1/3$. Then (2.5) leads to
\begin{align}
g_{1/3}(0)&=0,\quad g_{1/3}(1)=-\frac{\sqrt{2}}{3}<0,\\
g'_{1/3}(x)&=\frac{x^3}{\sqrt{1-x^4}}h_{1/3}(x),
\end{align}
where
\begin{equation}
h_{1/3}(x)=\frac{14}{9(\sqrt{1+x^2}+\sqrt{1-x^2})}-(\sqrt{1+x^2}+\sqrt{1-x^2})-\frac{\sqrt{1-x^2}}{3}.
\end{equation}
We clearly see that the function $\sqrt{1+x^2}+\sqrt{1-x^2}$ is strictly decreasing in $(0,1)$. Then from (2.8) we get
\begin{equation}
h_{1/3}(x)<h_{1/3}(1)=-\frac{2\sqrt{2}}{9}<0
\end{equation}
for $x\in(0,1)$.

Therefore, $f_{1/3}(x)<0$ for all $x\in(0,1)$ follows easily from (2.2), (2.4), (2.6), (2.7) and (2.9).

\item[Case 2]$p=\lambda_0$. Then (2.3) and (2.5) yield
\begin{equation}
 f_{\lambda_0}(1)=g_{\lambda_0}(0)=0,\quad g_{\lambda_0}(1)=-\sqrt{2}\lambda_0<0
\end{equation}
and
\begin{equation}
g'_{\lambda_0}(x)=\frac{x}{\sqrt{1-x^4}}h_{\lambda_0}(x),
\end{equation}
where
\begin{align}
\begin{split}
h_{\lambda_0}(x)=&[(2-3\lambda_0-2\lambda_0^2)-(3-6\lambda_0)x^2]\sqrt{1+x^2}\\
&-[(3\lambda_0-2\lambda_0^2)+(6\lambda_0-6\lambda_0^2)x^2]\sqrt{1-x^2}.
\end{split}
\end{align}

We divide the discussion of this case into two subcases.
\setlength\leftmargini{1cm}
\begin{description}
\item[Subcase A]$x\in(0.9,1)$. Then from (2.12) and the fact that
\begin{align*}
\lefteqn{(2-3\lambda_0-2\lambda_0^2)-(3-6\lambda_0)x^2}\\
&<(2-3\lambda_0-2\lambda_0^2)-(3-6\lambda_0)\times(0.9)^2=-0.1404\cdots<0
\end{align*}
we know that
\begin{equation}
h_{\lambda_0}(x)<0
\end{equation}
for $x\in(0.9,1)$.

\item[Subcase B]$x\in(0,0.9]$. Then from (2.12) one has
\begin{equation}
h_{\lambda_0}(0)=0.8137\cdots>0 ,\quad
h_{\lambda_0}(0.9)=-0.7494\cdots<0
\end{equation}
and
\begin{equation}
h'_{\lambda_0}(x)=\frac{x}{\sqrt{1-x^4}}\mu(x),
\end{equation}
where
\begin{align}\label{eqn:2.15}
\begin{split}
\mu(x)=&[(18\lambda_0-18\lambda_0^2)x^2-(9\lambda_0-10\lambda_0^2)]\sqrt{1+x^2}\\
&-[(9-18\lambda_0)x^2+(4-9\lambda_0+2\lambda_0^2)]\sqrt{1-x^2}.
\end{split}
\end{align}

We conclude that
\begin{equation}
\mu(t)<0
\end{equation}
for all $x\in(0,0.9]$. Indeed, if $x\in(0,1/2)$, then (2.17) follows from (2.16) and the inequality
$$(18\lambda_0-18\lambda_0^2)x^2-(9\lambda_0-10\lambda_0^2)<5.5\lambda_0^2-4.5\lambda_0=-0.6747\cdots<0.$$

If $x\in[1/2,0.9]$, then (2.17) follows from (2.16) and the inequalities
\begin{align*}
(18\lambda_0-18\lambda_0^2)x^2-(9\lambda_0-10\lambda_0^2)&\leq(18\lambda_0-18\lambda_0^2)\times(0.9)^2-(9\lambda_0-10\lambda_0^2)\\
&=5.58\lambda_0-4.58\lambda_0^2=0.9242\cdots,\\
(9-18\lambda_0)x^2+(4-9\lambda_0+2\lambda_0^2)&\geq\frac{1}{4}(9-18\lambda_0)+(4-9\lambda_0+2\lambda_0^2)\\
&=6.25-13.5\lambda_0+2\lambda_0^2=3.6589\cdots,\\
[(18\lambda_0-18\lambda_0^2)x^2-(9\lambda_0-10\lambda_0^2)]&\sqrt{1+x^2}\\
-[&(9-18\lambda_0)x^2+(4-9\lambda_0+2\lambda_0^2)]\sqrt{1-x^2}\\
\leq(5.58\lambda_0-4.58\lambda_0^2)\sqrt{1+(0.9)^2}-&(6.25-13.5\lambda_0+2\lambda_0^2)\sqrt{1-(0.9)^2}\\
&=-0.3514\cdots<0.
\end{align*}

From (2.14) and (2.15) together with (2.17) we clearly see that there exists $x_0\in(0,0.9)$ such that $h_{\lambda_0}(x)>0$ for $x\in[0,x_0)$
and $h_{\lambda_0}(x)<0$ for $(x_0,0.9]$.
\end{description}

Subcases A and B lead to the conclusion that $h_{\lambda_0}(x)>0$ for $x\in[0,x_0)$ and $h_{\lambda_0}(x)<0$ for $x\in(x_0,1)$. Thus from (2.11) we know that $g_{\lambda_0}(x)$ is strictly increasing in $(0,x_0]$ and strictly decreasing in $[x_0,1)$.

It follows from (2.4) and (2.10) together with the piecewise monotonicity of $g_{\lambda_0}(x)$ that there exists $x_1\in(0,1)$ such that $f_{\lambda_0}(x)$ is strictly increasing in $[0,x_1)$ and strictly decreasing in $[x_1,1)$.

Therefore, $f_{\lambda_0}(x)>0$ for $x\in(0,1)$ follows from (2.2) and (2.10) together with the piecewise monotonicity of $f_{\lambda_0}(x)$.
\end{description}
\end{proof}

\section{Proof of Theorem 1.1-1.3}

\medskip

\begin{proof}[\bf Proof of Theorem 1.1]
Since $H(a,b)$, $M(a,b)$ and $Q(a,b)$ are symmetric and homogeneous of degree 1. Without loss of generality, we assume that $a>b$. Let
$x=(a-b)/(a+b)$ and $t=\sinh^{-1}(x)$. Then $x\in(0,1)$, $t\in(0,\log(1+\sqrt{2}))$, $M(a,b)/A(a,b)=x/\sinh^{-1}(x)=\sinh(t)/t$,
$H(a,b)/A(a,b)=1-x^2=1-\sinh^2(t)=[3-\cosh(2t)]/2$, $Q(a,b)/A(a,b)=\sqrt{1+x^2}=\cosh(t)$ and
\begin{align}
\begin{split}
\frac{Q(a,b)-M(a,b)}{Q(a,b)-H(a,b)}&=\frac{\sqrt{1+x^2}\sinh^{-1}(x)-x}{[\sqrt{1+x^2}-(1-x^2)]\sinh^{-1}(x)}\\
&=\frac{t\cosh(t)-\sinh(t)}{t[\frac{1}{2}\cosh(2t)+\cosh(t)-\frac{3}{2}]}:=\varphi(t).
\end{split}
\end{align}

Making use of power series $\sinh(t)=\sum_{n=0}^{\infty}t^{2n+1}/(2n+1)!$ and
$\cosh(t)=\sum_{n=0}^{\infty}t^{2n}/(2n)!$ we can express (3.1) as follows
\begin{equation}
\varphi(t)=\frac{\sum_{n=1}^{\infty}[2n/((2n+1)(2n)!)]t^{2n+1}}{\sum_{n=1}^{\infty}[(2^{2n-1}+1)/(2n)!]t^{2n+1}}.
\end{equation}
Let $a_n=2n/((2n+1)(2n)!)$ and $b_n=(2^{2n-1}+1)/(2n)!$. Then
$a_n/b_n=2n/(2n+1)(2^{2n-1}+1)$. Moreover, by a simple calculation,
we see that
\begin{equation}
\frac{a_{n+1}}{b_{n+1}}-\frac{a_n}{b_n}=\frac{2+(2-18n-12n^2)2^{2n-1}}{(2n+1)(2n+3)(2^{2n-1}+1)(2^{2n+1}+1)}<0
\end{equation}
for $n\geq1$.

Equations (3.1) and (3.2) together with inequality (3.3) and Lemma 2.1 lead to the conclusion that $\varphi(t)$ is strictly decreasing in
$(0,\log(1+\sqrt{2}))$. This in turn implies that
\begin{equation}
\lim\limits_{t\rightarrow0^+}\varphi(t)=\frac{2}{9}, \quad \lim\limits_{t\rightarrow \log(1+\sqrt{2})}\varphi(t)=1-\frac{1}{\sqrt{2}\log(1+\sqrt{2})}.
\end{equation}

Therefore, Theorem 1.1 follows from (3.1) and (3.4) together with the monotonicity of $\varphi(t)$.
\end{proof}

\medskip

\begin{proof}[\bf Proof of Theorem 1.2]
Since $G(a,b)$, $M(a,b)$ and $Q(a,b)$ are symmetric and homogeneous of degree 1. Without loss of generality, we assume that $a>b$. Let
$x=(a-b)/(a+b)$, $p\in(0,1)$ and $\lambda_0=1-1/[\sqrt{2}\log(1+\sqrt{2})]$. Then making use of $G(a,b)/A(a,b)=\sqrt{1-x^2}$ gives
\begin{equation}
\frac{Q(a,b)-M(a,b)}{Q(a,b)-G(a,b)}=\frac{\sqrt{1+x^2}\sinh^{-1}(x)-x}{(\sqrt{1+x^2}-\sqrt{1-x^2})\sinh^{-1}(x)}.
\end{equation}
Moreover, we obtain
\begin{align}
\lim\limits_{x\rightarrow0^+}\frac{\sqrt{1+x^2}\sinh^{-1}(x)-x}{(\sqrt{1+x^2}-\sqrt{1-x^2})\sinh^{-1}(x)}&=\frac{1}{3},\\
\lim\limits_{x\rightarrow1^-}\frac{\sqrt{1+x^2}\sinh^{-1}(x)-x}{(\sqrt{1+x^2}-\sqrt{1-x^2})\sinh^{-1}(x)}&=1-\frac{1}{\sqrt{2}\log(1+\sqrt{2})}=\lambda_0.
\end{align}

We take the difference between the additive convex combination of $G(a,b),Q(a,b)$ and $M(a,b)$ as follows
\begin{align}\label{eqn:3.10}
\begin{split}
\lefteqn{pG(a,b)+(1-p)Q(a,b)-M(a,b)}\\
&=A(a,b)\left[p\sqrt{1-x^2}+(1-p)\sqrt{1+x^2}-\frac{x}{\sinh^{-1}(x)}\right]\\
&=\frac{A(a,b)[p\sqrt{1-x^2}+(1-p)\sqrt{1+x^2}]}{\sinh^{-1}(x)}f_p(x),
\end{split}
\end{align}
where $f_p(x)$ is defined as in Lemma 2.2.

\medskip

Therefore, $\frac{1}{3}G(a,b)+\frac{2}{3}Q(a,b)<M(a,b)<\lambda_0G(a,b)+(1-\lambda_0)Q(a,b)$ for all $a,b>0$ with $a\neq b$ follows from (3.8) and
Lemma 2.2. This in conjunction with the following statement gives the asserted result.
\setlength\leftmargini{.3cm}
\begin{itemize}
\item If $p<1/3$, then equations (3.5) and (3.6) imply that there exists $0<\delta_1<1$ such that $M(a,b)<p G(a,b)+(1-p)Q(a,b)$ for all $a,b>0$ with $(a-b)/(a+b)\in(0,\delta_1)$.

\item If $p>\lambda_0$, then equations (3.5) and (3.7) imply that there exists $0<\delta_2<1$ such that $M(a,b)>p G(a,b)+(1-p)Q(a,b)$ for all $a,b>0$ with $(a-b)/(a+b)\in(1-\delta_2,1)$.
\end{itemize}
\end{proof}

\begin{proof}[\bf Proof of Theorem 1.3]
We will follow, to some extent, lines in the proof of Theorem 3.1.
First we rearrange terms of (1.3) to obtain
\begin{equation*}
\beta_3<\frac{C(a,b)-M(a,b)}{C(a,b)-H(a,b)}<\alpha_3.
\end{equation*}
Use of $C(a,b)/A(a,b)=1+x^2$ followed by a substitution $x=\sinh(t)$
gives
\begin{equation}
\beta_3<\phi(t)<\alpha_3
\end{equation}
where
\begin{equation}
\phi(t)=\frac{t[\cosh(2t)+1]-2\sinh(t)}{2t[\cosh(2t)-1]},\quad
|t|<\log(1+\sqrt{2}).
\end{equation}

Since the function $\phi(t)$ is an even function, it suffices to investigate its
behavior on the interval $(0, \log(1+\sqrt{2}))$.

\medskip

Using power series of $\sinh(t)$ and $\cosh(t)$, then (3.10) can be rewritten as
\begin{equation}
\phi(t)=\frac{\sum_{n=1}^{\infty}[2^{2n}/(2n)!-2/(2n+1)!]t^{2n+1}}{\sum_{n=1}^{\infty}[2^{2n+1}/(2n)!]t^{2n+1}}.
\end{equation}

Let $c_n=2^{2n}/(2n)!-2/(2n+1)!$ and $d_n=2^{2n+1}/(2n)!$. Then
\begin{equation}
\frac{c_n}{d_n}=\frac{1}{2}-\frac{1}{(2n+1)2^{2n}}.
\end{equation}
It follows from (3.12) that the sequence $\{c_n/d_n\}$ is strictly increasing for $n\geq1$.

\medskip

Equations (3.11) and (3.12) together with Lemma 2.1 and the
monotonicity of $\{c_n/d_n\}$ lead to the conclusion that $\phi(t)$
is strictly increasing in $(0,\log(1+\sqrt{2}))$. Moreover,
\begin{equation}
\lim\limits_{t\rightarrow0^+}\phi(t)=\frac{c_1}{d_1}=\frac{5}{12},\quad
\lim\limits_{t\rightarrow\log(1+\sqrt{2})}\phi(t)=1-\frac{1}{2\log(1+\sqrt{2})}.
\end{equation}

Making use of (3.13) and (3.9) together with the monotonicity of
$\phi(t)$ gives the asserted result.

\end{proof}

\bigskip
{\small \hfill {\sc Tiehong Zhao}

\hfill {\sl Department of Mathematics}

\hfill {\sl Hangzhou Normal University}

\hfill {\sl Hangzhou 310036}

\hfill {\sl China}

\hfill {\sl Email:} \texttt{tiehongzhao@gmail.com}

\bigskip
{\small \hfill {\sc Yuming Chu}

\hfill {\sl Department of Mathematics}

\hfill {\sl Huzhou Teachers College}

\hfill {\sl Huzhou 313000}

\hfill {\sl China}

\hfill {\sl Email:} \texttt{chuyuming2005@yahoo.com.cn}

\bigskip
{\small \hfill {\sc Baoyu Liu}

\hfill {\sl School of Science}

\hfill {\sl Hangzhou Dianzi University}

\hfill {\sl Hangzhou 310018}

\hfill {\sl China}

\hfill {\sl Email:} \texttt{627847649@qq.com}
\end{document}